\documentclass[12pt]{amsart}

\def\MMR#1{\relax}

\usepackage[latin1]{inputenc}
\usepackage{amssymb,amsmath,amsfonts,latexsym,graphics,graphicx,amsthm}
\usepackage{supertabular,float}
\usepackage[OT2,OT1]{fontenc}
\newcommand{\cyr}{%
\renewcommand\rmdefault{wncyr}%
\renewcommand\sfdefault{wncyss}%
\renewcommand\encodingdefault{OT2}
\normalfont\selectfont}
\DeclareTextFontCommand{\textcyr}{\cyr}
\newtheorem{theorem}{Theorem}[section]

\newtheorem{definition}[theorem]{Definition} 

\newtheorem{proposition}[theorem]{Proposition}

\newtheorem{remark}[theorem]{Remark}


\date{February 20, 2010}

\title{Wagner Lift of Riemannian metric to Orthogonal Frame Bundle}

\author{Mikhail Malakhaltsev}
\address{Department of Mathematics, Kazan State University, 
Department of Mathematics, Ulitza Kremlovskaia, Dom 8, Kazan, Russia}
\email{mikhail.malakhaltsev@ksu.ru}
\author{José Ricardo Arteaga B.}
\address{Department of Mathematics, Universidad de los Andes,
Cra. 1 ESTE 18A-10, Bogotá, Colombia}
\email{jarteaga@uniandes.edu.co}
\begin{document}
\maketitle

%-----------------------
%-----------------------
\begin{abstract}
In the present work we construct  a lift of a metric $g$ on a 2-dimensional oriented Riemannian manifold
$M$ to a metric $\hat{g}$ on the total space $P$ of the orthonormal frame bundle of $M$.
We call this lift the \textit {Wagner lift}. Viktor Vladimirovich Wagner (1908 -1981) proposed a technique to extend a
metric defined on a non-holonomic distribution to its prolongation via the Lie brackets. 
We apply the Wagner construction to the specific case when the distribution is the infinitesimal connection in the orthonormal frame bundle
which corresponds to a Levi-Civita connection.
We find relations between the geometry of the Riemannian manifold $(M,g)$ and of the total space $(P,G)$ of the orthonormal frame bundle endowed
with the lifted metric.

\end{abstract}

%-----------------------section 1: Introducción
\section{Introduction}

A sub-Riemannian manifold $(\mathcal{H}, g)$ is a distribution $ \mathcal{H}$ on a manifold $ Q $ endowed by a metric tensor $g$ \cite{Montgomery}.

V.V. Wagner \cite{Wagner} was one of the first mathematicians who considered the geometry of a nonholonomic distribution endowed by metric tensor. One of his main result was the construction which made it possible to extend a metric given on a distribution $\mathcal{H}$ to the manifold $ Q $. This extension was made via the Lie bracket of vector fields.

Let us describe this construction in case $\dim Q = n$ and $\dim \mathcal{H}=2$.
Locally ${\mathcal{H}}$ can be given as follows:
\begin{equation}
{\mathcal{H}}= span\left\{E_1, E_2\right\},\notag
\end{equation} 
where $E_{1}, E_{2}$ are linearly independent vector fields such that the Lie bracket
\begin{equation}
[E_{1}, E_{2}] \notin {\mathcal{H}}.\notag
\end{equation}
If $n\geq 3$, we can set 
\begin{equation}
{\mathcal{H}^{1}}=span\left\{E_1, E_2, E_{3}\right\},\quad \text{where}\quad E_{3}=[E_{1}, E_{2}].\notag
\end{equation}
The distribution ${\mathcal{H}^{1}}$ is called the \textit{first prolongation} of ${\mathcal{H}}$ via the Lie bracket. 
If
\begin{equation}
[E_{i}, E_{j}] \notin \mathcal{H}^{1}, \, i,j=1,2,3,\notag
\end{equation}
we can define 
\begin{equation}
{\mathcal{H}^{2}}=span\left\{E_1, E_2, E_{3}, E_{4}, E_{5}\right\},\quad \text{where}\quad E_{4}=[E_{1}, E_{3}] \, \text{y}\, E_{5}=[E_{2}, E_{3}].\notag
\end{equation}
$\mathcal{H}^{2}$ is called the second prolongation of  $\mathcal{H}$. In this manner, step-by-step one can construct a series of prolongations via the Lie brackets.

For a given metric tensor $g$ on $\mathcal{H}$ and a framing $\mathcal{N}$ of $\mathcal{H}$ (i.\,e. $TQ = \mathcal{H}\oplus \mathcal{N}$),  Wagner extended the metric $g$ on $\mathcal{H}$ to metrics on the prolongations ${\mathcal{H}^{k}}$ \cite{Wagner}. For simplicity, we will explain this construction for the case of completely non-holonomic distributions, for which $\mathcal{H}^{1}=TQ$. 
Wagner introduces the non-holonomity  tensor 
$N : \mathcal{H} \times \mathcal{H} \to \mathcal{N}$, $N(X,Y) = pr_\mathcal{N}([X,Y])$, where $pr_\mathcal{N}$ is the projector onto $\mathcal{N}$ with respect to $\mathcal{H}$. For non-holonomic distributions $N$ is surjective. The metric $g$ on $\mathcal{H}$ induces the metric $g^\Lambda$ on the space $\Lambda^2(\mathcal{H})$ of bivectors on $\mathcal{H}$ in the standard way. Then,
$\Lambda^2(\mathcal{H}) \cong \ker N \oplus (\ker N)^\perp$, and in this case $(\ker N)^\perp \cong \mathcal{N}$. The metric $g^\Lambda$ induces a metric on $(\ker N)^\perp \subset \Lambda^2(\mathcal{H})$ and this metric is transferred to a metric $g_\mathcal{N}$ on $\mathcal{N}$. Thus, we have the metric $g$ on $\mathcal{H}$ and the metric $g_\mathcal{N}$ on 
$\mathcal{N}$, and thus we get the metric $g^{1}$ on $\mathcal{H}^1=TQ$ assuming that $\mathcal{H}$ is orthogonal to $\mathcal{N}$.

Now consider the case $\dim Q = 3$ and $\dim\mathcal{H}=2$. In this case, if $\mathcal{H}$ is completely non-holonomic, $TQ = \mathcal{H}^{1}$. Now take an orthonormal frame $E_1, E_2$ of $\mathcal{H}$ with respect to $g$, and set $E_3 = N(E_1,E_2)$. Then, the metric $g^1$ on $TQ=\mathcal{H}^{1}$ can be uniquely determined by the condition that $\{E_1, E_2,E_3\}$ is an orthonormal frame of $TQ$. 

In the present paper we consider a 2-dimensional Riemannian manifold $(M,g)$, and take the orthonormal frame bundle $P\longrightarrow M$. The metric $g$ determines the Levi-Civita connection on $M$ which, in turn, determines the horizontal distribution $\mathcal{H}$ on $P$. We take the lift $g_\mathcal{H}$ of $g$ to $\mathcal{H}$ and thus obtain a sub-Riemannian manifold $(\mathcal{H},g_\mathcal{H})$. For the framing $N$ of $\mathcal{H}$ we take the vertical distribution (the distribution tangent to fibres of $P \to M$).  And thus, using Wagner's construction, we define a metric $\hat g$ on $TP$ and call it the \textit{Wagner lift} of the metric $g$ on $M$.

% % % % TO TRANSLATE % %
La organización de las sectiones siguientes es como sigue. 
La sección 2 introduce los temas preliminares necesarios de conexiones en haces principales, lo cual todo se puede encontrar en \cite{Montgomery02} and \cite{KN}.  
La sección 3 muestra el levantamiento de una distribution no holonómica and del tensor no holonómico, cuando el espacio  base is una variedad de Riemann bi-dimensional. 
La sección 4 describe lo que hemos denominado el lavantamiento de Wagner de una métrica bidimensional al haz de principal de los marcos ortonormales.
La sección 5 muestra las relationes entre los objetos geométricos de la variedad base and del haz principal.
Finalmente en la sección 6 se muestran dos ejemplos en los cuales se muestra las técnicas de levantamiento and además exhibe algunos de los resultados mostrados en este artículo.
 
%--------------------------------- Sección 2: Preliminares
\section{Preliminares}

Here we present a brief review of the theory of connections in principal bundle which we will use throughout the paper (for the details we refer the reader to \cite{KN}, Vol. I). 

Let $\pi : P \to M$ be a $G$-principal bundle over a smooth manifold, where $G$ is a Lie group. 
We will denote the action of $G$ on the total space $P$ as follows
\begin{eqnarray}
P\times G &\longrightarrow& P \notag\\
(u,a) &\mapsto& ua = R_{a}u  \notag
\end{eqnarray}
The fibres of $P(M,G)$ are the orbits of the group action.

The action of  $G$ on $P$ induces a homomorphism $\sigma$ of the Lie algebra $\mathfrak{g}$ of $G$ to the Lie algebra $\mathfrak{X}(P)$ of vector fields on $P$:
\begin{equation}
\sigma : \mathfrak{g} \longrightarrow \mathfrak{X}(P), \quad \sigma (A)_u = \dfrac{d}{dt}\big|_{t=0} u \exp(tA).
\label{eq:2_1}
\end{equation} 
The vector fields $\sigma (A)$ are called \textit{fundamental vector fields}.

Let us denote by $G_{u}$ the subspace in the tangent space $T_{u}P$ consisting of vectors tangent to the fibres of the bundle, call it the \textit{vertical subspace} at $u$. Note that $G_u = span \left\{\sigma (A)_u \, | \,  A \in  \mathfrak{g} \right\}$.

Recall that the infinitesimal connection $\Gamma$ on $P$ is a distribution $\mathcal{H}$ such that
\begin{enumerate}
\item $T_{u}P=\mathcal{H}_{u}\oplus G_{u}$,
\item $\mathcal{H}_{ua}= (R_{a})_{*}\mathcal{H}_{u}$.
\end{enumerate}
We will call $\mathcal{H}_{u}$ the \textit{horizontal subspace}. A curve $\tilde{\gamma}(t)$ in $P$ such that 
the tangent vector $\frac{d}{dt}\tilde{\gamma}(t)$ lies in the horizontal subspace $\mathcal{H}_{\gamma(t)}$ is said to be 
\textit{horizontal} 

If $\gamma(t)$, $t \in [a,b]$ is a curve in $M$, then, for each $u \in \pi^{-1}(\gamma(a))$, there exists a unique horizontal curve $\tilde{\gamma}(t)$ such that $\tilde{\gamma}(a)=u$ and $\pi(\tilde{\gamma}(t))=\gamma(t)$. The curve 
$\tilde{\gamma}(t)$ is called a \textit{horizontal lift} of $\gamma$.

Each vector field $X$ on $M$ determines uniquely a vector field $X^h$ on $P$ such that $X^h(u) \in \mathcal{H}_u$  and $d\pi(X^h(u) = X(\pi(u))$, for each $u \in P$. The vector field  $X^h$ is called the \textit{horizontal lift} of $X$.

Now let us consider the orthonormal frame principal bundle $SO(M)$ of an oriented Riemannian manifold $(M,g)$. 
Then $G=SO(n)$ and its Lie algebra is the Lie algebra ${\mathfrak{o}}(n)$ of the antisymmetric matrices.
The \textit{canonical form} $\theta$ on $SO(M)$ is a 1-form with values in $\mathbf{R}^{n}$ 
($\theta : T SO(M) \longrightarrow \mathbf{R}^{n}$) defined by
\begin{equation}
\theta (X) = u^{-1}(d\pi (X)) \quad \text{for all} \, X\in T_{u}SO(M), \notag
\end{equation} 
where $u$ is considered as an isomorphism between the vector spaces $\mathbf{R}^{n}$ and 
$T_{\pi (u)}M$, $u : \mathbf{R}^{n}\longrightarrow T_{\pi(u)}M$.
Each $\xi \in \mathbf{R}^{n}$ determines a vector field $B(\xi )$ on $SO(M)$ in the following way. For each $u\in SO(M)$, the vector $(B(\xi))_{u}$ is the unique  horizontal vector in $u$ such that 
\begin{equation}\label{eq:2_2}
d\pi ((B(\xi))_{u}) = u(\xi) = \sum_{i=1}^{n}\xi^{i}e_{i},
\end{equation}
where $\xi = (\xi^{1},\dots , \xi^{n} )\in \mathbf{R}^{n}$ and $u=\left\{e_{1}, \dots , e_{n}\right\}$ is the orthonormal frame of $T_{x}M$. $B(\xi)$ is called the \textit{standard horizontal vector field} corresponding to $\xi \in \mathbb{R^n}$.
We will denote by $\{B_i\}$ the standard horizontal vector fields corresponding to the standard frame vectors  
$e_i = (0,0,\dots , 0, 1, 0, \dots , 0)$ in  $\mathbb{R}^n$. 

The standard horizontal vector fields satisfy the following properties:
\begin{enumerate}
\item $\theta (B(\xi))=\xi \quad \forall \xi$,
\item $R_{a}(B(\xi))= B(a^{-1}\xi)$,
\item $B(\xi)$ does not vanish if $\xi \neq \overrightarrow{0}$.
\end{enumerate}

Let $\nabla$ be the Levi-Civita connection of $g$. Then $\nabla$ determines an infinitesimal connection ${\mathcal{H}}$ on $P$ \footnote{One of the first mathematicians who studied infinitesimal connections associated with Levi-Civita connections was  A.P. Shirokov, see \cite{Shirokov67}}. The horizontal lift of a curve $\gamma$ in $M$ is the field of orthonormal frames which is parallely translated along $\gamma$ with respect to the  Levi-Civita connection.

This distribution ${\mathcal{H}}$ is the kernel of the \emph{connection form} $\omega : T SO(M) \to \mathfrak{g}$.  
If $U\subset M$ is an open set such that $SO(M)|_U \cong U \times G$ ($SO(M)|_U$ is trivial), 
then on $SO(M)|_U$ we have ``coordinates'' $(x,\mu)$, $x \in M$, $\mu \in G$. With respect to these coordinates the connection form is written as follows
\begin{equation}\label{eq:2_3}
\omega^k_l = \widetilde{\mu}^{k}_{s} (d\mu^{s}_{l} + \Gamma^{s}_{ij} \mu^{j}_{l} \theta^{i}), \notag
\end{equation}
where $\Gamma^{k}_{ij}$ are the coefficients of the Levi-Civita connection with respect to the frame field determined by the  trivialization, $\theta^{i}$ are the components of the canonical form $\theta$, and $\widetilde{\mu}$ is the matrix inverse to  $\mu$.

The \textit{curvature form} $\Omega : \Lambda_2 T SO(M) \to \mathfrak{g}$ of the connection form $\omega$ is given with respect to the above trivialization as follows:
\begin{equation}
\Omega_{k}^{l}(p) = R_{ijk}^l \theta^i \wedge \theta^j, \notag
\end{equation}
where $R_{ijk}^l$ are the components of the curvature of the Levi-Civita connection.

%----------------section 3: Levantamiento de una distribution no holonomica
\section{El Levantamiento de una distribution no Holonómica}

%----------------subsection 3.5 El Haz Principal $\mathbb{SO}(M)$
\subsection[section]{Calculations in dimension 2}

Let us consider an oriented 2-dimensional manifold $M$ endowed by a Riemannian metric $g$. 
In this case, $G=SO(2)$ and we will denote its elements by
\begin{equation}
R(\varphi) = 
\left(
\begin{array}{cc}
\cos \varphi & -\sin \varphi
\\
\sin \varphi & \cos \varphi
\end{array}
\right) \notag
\end{equation}

Then, \eqref{eq:2_3} for the connection form $\omega$ on $SO(M)$ determined by the Levi-Civita connection $\nabla$ is written as follows:
\begin{equation}
\omega_{(x,\varphi)} = R(\varphi)^{-1}
 \left(\frac{dR(\varphi)}{d\varphi}d\varphi + \Gamma_a R(\varphi)\theta^a\right).
\label{eq:3_4}
\end{equation}
Here $\Gamma_a$ is a 1-form with values in the Lie algebra $\mathfrak{o}(2)$ whose entries are coefficients of $\nabla$ with respect to an orthonormal frame field:
\begin{equation}
\nabla_{e_a} e_b = \Gamma_{ab}^{c} e_{c}\, ; \quad 
\Gamma_{a}=
\left(
\begin{array}{cc}
  0 & \Gamma_{a2}^{1}\MMR{\Gamma_{a2}^{1}}\\
\Gamma_{a1}^{2} \MMR{\Gamma_{a1}^{2}}& 0
\end{array}
%%
% Certainly!
%%
\right) \notag
\end{equation}
The matrix $\Gamma_a = ||\Gamma_a{}^c_b||$ is antisymmetric, that is $\Gamma_{ab}^{c}=-\Gamma_{ac}^{b}$, since $e$ is an orthonormal frame and $\nabla$ is the Levi-Civita connection. 

As $SO(2)$ is a commutative group we have that
$R(\varphi)^{-1} \Gamma_a R(\varphi)= \Gamma_a$. 
In addition,  $\frac{dR}{d\varphi} = R(\varphi+\pi/2)$, so 
\begin{equation}
R^{-1}(\varphi) \frac{dR}{d\varphi} = R(\pi/2) = 
J = 
\left(
\begin{array}{cc}
0 & -1
\\
1 & 0
\end{array}
\right). \notag
\end{equation}

Now we can rewrite \eqref{eq:3_4} as follows: 
\begin{equation}
\omega = J d\varphi + \Gamma_a \theta^a, \notag
\end{equation}
or, in the matrix form, as  
\begin{equation}
\left(
\begin{array}{cc}
0 & \alpha
\\
-\alpha & 0
\end{array}
\right)
= 
\left(
\begin{array}{cc}
0 & d\varphi 
\\
-d\varphi & 0
\end{array}
\right)
+
\left(
\begin{array}{cc}
0 & \Gamma_1{}^1_2 \theta^1
\\
-\Gamma_1{}^1_2 \theta^1 & 0
\end{array}
\right)
+
\left(
\begin{array}{cc}
0 & \Gamma_2{}^1_2 \theta^2
\\
-\Gamma_2{}^1_2\theta^2 & 0
\end{array}
\right)\notag
\end{equation}
Hence  
\begin{equation}
\alpha = d\varphi +\Gamma_{12}^{1} \theta^{1} + \Gamma_{22}^{1}\theta^{2}
\label{eq:3_5}
\end{equation}
where $\Gamma_{ab}^{c}$ are the coefficients of $\nabla$ and $\theta=\left\{\theta^{1}, \theta^{2}\right\}$ are components of the canonical form on $SO(M)$.
 
%------------------------------subsection 3.7 Los sub-haces tangentes: Horizontal H and Vertical V  
%\subsection{Los Sub-haces Tangentes: Horizontal $\mathcal{H}$ and Vertical $V$}

%---------------------- Proposition 3.1

%\begin{proposition}\label{prop:3_1}
Let $e=\left\{e_{1}, e_{2}\right\}$ be a local orthonormal frame field on $M$. Then, in terms of the trivialization of $SO(M)$ determined by $e$, the horizontal lifts of the vector fields  $e_{1}$, $e_{2}$ can be written as follows:
\begin{equation}
E^h_1(x,\varphi) = e_1 - \Gamma_1{}^1_2 \partial_\varphi,\quad 
E^h_2(x,\varphi) = e_2 - \Gamma_2{}^1_2 \partial_\varphi.
\label{eq:3_6}
\end{equation}
%\end{proposition}

%\begin{proof}
This follows immediately from \eqref{eq:3_5} and the fact that $\omega(E_{i}^{h})=0$.
%\end{proof}

Note that  
\begin{equation}
\mathcal{H}=\ker \omega = span \left\{E_{1}^{h}, E_{2}^{h}\right\}\notag
\end{equation}
and the standard horizontal vector fields $B_1$, $B_2$ can expressed in terms of $E^h_1$, $E^h_2$ as follows:
\begin{equation}
\{B_1(x,\varphi),B_2(x,\varphi)\} = \{E^h_1(x,\varphi),E^h_2(x,\varphi)\}R(\varphi)\notag
\end{equation}
because $\{d\pi(B_1(x,\varphi)),d\pi(B_2(x,\varphi))\} = \{e_1,e_2\} R(\varphi)$.

%------------- Proposition 3.2
%\begin{proposition}\label{prop:3_2}
Let $A = \left(
\begin{array}{cc}
0 & -m
\\
m & 0
\end{array}
\right)$, $m\in \mathbf{R}$. 
Then the fundamental vector field 
$\sigma( A ) = m \partial_\varphi$.
%\end{proposition}

%\begin{proof}
This follows from the definition of fundamental vector fields (see \eqref{eq:2_1}) and the simple calculation.
Indeed, $A=mJ$, $\exp (A\cdot t)= R(mt)$, where $J^{2}=-Id$. Hence 
\begin{equation}
\sigma (A) = \dfrac{d}{dt} u (x, \varphi )\cdot \exp (A\cdot t)\big|_{t=0} = m \partial\varphi.\notag
\end{equation}
%\end{proof}

%----------------------- subseciion 3.8 Expresiones via funtiones estructurales
\subsection{Expressions via structure functions}

For an orthonormal frame field  $\{e_i\}$ one can define the structure functions  $c^k_{ij}$:
\begin{equation}
[e_i,e_j] = c^k_{ij} e_k. \notag
\end{equation}
Note that the coefficients of the Levi-Civita connection can expressed in terms the structure functions.
For the Levi-Civita connection $\nabla$ we have (\cite{KN}, pág. 160) 
\begin{eqnarray}\label{eq:3_7}
g(\nabla_X Y,Z) &=& \frac{1}{2}(X g(Y,Z) + Y g(X,Z) - Z g(X,Y)\nonumber\\ 
						&+& g([X,Y],Z) + g([Z,X],Y) - g([Y,Z],X)).
\end{eqnarray}
In \eqref{eq:3_7}, we set $X=e_i$, $Y=e_j$, $Z=e_k$, and use $\nabla_{e_{i}}e_{j}=\Gamma_{ij}^{k}e_{k}$, then we obtain
that the coefficients of $\nabla$ with respect to the orthonormal frame $\{e_i\}$ can be written as follows: 
\begin{equation}
\Gamma^k_{ij} = \frac{1}{2}(c^k_{ij} + c^j_{ki} + c^i_{kj})
\label{eq:3_8}
\end{equation}
For the curvature tensor 
\begin{equation}\label{eq:3_9}
R(X,Y)Z = \nabla_{X}\nabla_{Y}Z - \nabla_{Y}\nabla_{X}Z - \nabla_{[X,Y]}Z
\end{equation}
we have
\begin{equation}\label{eq:3_11}
R_{ijk}^l = e_i\Gamma^{l}_{jk} - e_j\Gamma^{l}_{ik} + \Gamma^{l}_{is}\Gamma^{s}_{jk} 
- \Gamma^{l}_{js}\Gamma^{s}_{ik} - c^{s}_{ij} \Gamma^{l}_{sk},
\end{equation}
this gives us components $R_{ijk}^l$  in terms of $c^k_{ij}$.

On the other hand, \eqref{eq:3_6} can be rewritten in term of structure functions as follows:
\begin{equation}\label{eq:2_13}
E^h_1(x,\varphi) = e_1 - c^1_{12}(x) \partial_\varphi,
E^h_2(x,\varphi) = e_2 - c^2_{12}(x) \partial_\varphi.
\end{equation}

%----------------------- subsection 3.9 El tensor No-holonomico
\subsection{Nonholonomity tensor}

Let us consider the nonholonomity tensor $ N : \Lambda^{2}(\mathcal{H}) \to V$, where $\Lambda^{2}(\mathcal{H}) $ is the space of bivectors on the the distribution. By definition,
\begin{equation}\label{eq:3_13}
N(X,Y) = proj_V ([X,Y]).
\end{equation}

Note that  
\begin{equation}\label{eq:3_14}
\begin{split}
[E^h_1,E^h_2] = [e_1 - c^1_{12}(x) \partial_\varphi,e_2 - c^2_{12}(x) \partial_\varphi] 
= 
[e_1,e_2] + (e_2 c^1_{12}(x) - e_1 c^2_{12}(x)) \partial_\varphi
\\
=
c^1_{12}(x) e_1 + c^2_{12}(x) e_2 + (e_2 c^1_{12}(x) - e_1 c^2_{12}(x)) \partial_\varphi 
\end{split}
\end{equation}
From \eqref{eq:2_13} we find $e_{1}$, $e_{2}$ and  substitute them in (\ref{eq:3_14}), thus we obtain the following expression for the Lie bracket $[E^h_1,E^h_2] $:
\begin{equation}\label{eq:3_15}
[E^h_1,E^h_2] = 
c^1_{12} E^h_1 + c^2_{12} E^h_2 + (e_2 c^1_{12}(x) - e_1 c^2_{12}(x) 
+ (c^1_{12}(x))^2 + (c^2_{12}(x))^2) \partial_\varphi.
\end{equation} 
Thus we get 

%--------------------------------Proposicin 3.3
\begin{proposition}\label{prop:3_3}
Let $e=\left\{e_{1}, e_{2}\right\}$ be an orthonormal frame field on the base $M$, and $E^{h}=\left\{E_{1}^{h}, E_{2}^{h}\right\}$ are the horizontal lifts to $SO(M)$ defined by \eqref{eq:2_13}. Then the nonholonomicity tensor $N$ satisfies
\begin{equation}\label{eq:3_16}
N(E^h_1(x,\varphi),E^h_2(x,\varphi)) = - K(x) \partial_\varphi
\end{equation} 
where $K(x)$ is the curvature of $(M,g)$ at $x \in M$. 
\end{proposition}
\begin{proof}
By the definition of $K$, we have that  
\begin{equation}
K=g(R(e_1,e_2)e_2,e_1)= R^1_{122} \notag
\end{equation} 
where $e_1$, $e_2$ is an orthonormal frame.
Then, using \eqref{eq:3_8}  and \eqref{eq:3_11}, we obtain that  
\begin{equation}\label{eq:3_17}
K(x) = R^1_{122} =   
e_1 c^2_{12}(x) - e_2 c^1_{12}(x) - (c^1_{12}(x))^2 - (c^2_{12}(x))^2.
\end{equation}  
Hence, by \eqref{eq:3_13} and \eqref{eq:3_15}, we get the required statement.
\end{proof}

% % % % % % 30 01 2010
%----------------------- section 4: Levantamiento de Wagner de la metrica
\section{Wagner lift of a metric metric $g$ on $M$ to $\hat{g}$ en $P$}

%----------------definition 4.1
We need to define $\hat{g}(X, Y)$ for any smooth vector fields on $P$. 
We have $TP=\mathcal{H}\oplus V$, hence any vector field can be expressed in terms of horizontal vector fields and fundamental vector fields. 
\begin{definition}\label{def:4_1}
The Wagner of a metric is defined as follows:
\begin{enumerate}
\item 
\begin{equation}
\hat g(E^h_a,E^h_b) = g(d\pi(E^h_a),d\pi(E^h_b))=g(e_a,e_b)=\delta_{ab} \notag
\label{eq:4_1}
\end{equation}
\item
\begin{equation}
\hat g(E^h_a,\partial_\varphi) = 0. \notag
\label{eq:4_2}
\end{equation}
\item 
\begin{equation}
\hat g(K(x)\partial_\varphi,K(x)\partial_\varphi)=1. \notag
\label{eq:4_3}
\end{equation} 
\end{enumerate}
\end{definition}
\begin{remark}
The frame $E^{h}=\left\{E_{1}^{h}, E_{2}^{h}\right\}$ is orthonormal with respect to $\hat{g}$. 
\end{remark}
\begin{remark}
The definition of $\hat g$ on the vertical vector fields follows from Wagner's construction for extension of metric (see Introduction and Proposition~\ref{prop:3_3}). 
\end{remark}
\begin{remark}
The equation \eqref{eq:4_3} shows that our metric can be defined only for 
$x\in M$ where $K(x)\neq 0$. If there exist points $x\in M$ such that $K(x)=0$, we say that 
the metric $\hat{g}$ has singularities along the fibers over these points. 
In what follows we consider two-dimensional Riemannian manifolds $M$ such that $K(x)\neq 0$ for all $x\in M$.
\end{remark}
%----------------------------Definition 4.2
\begin{definition}\label{def:4_2}
Let $e=\{e_{1}(x), e_{2}(x)\}$ be an orthonormal frame defined on an open set $U\subset M$. On $P=SO(M)$ we define the frame $\{\mathcal{E}_1,\mathcal{E}_2,\mathcal{E}_3\}$, which is orthonormal with respect to $\hat g$, as follows:
\begin{equation}
\begin{array}{l}\label{eq:4_18}
\mathcal{E}_1(x,\varphi) = E^h_1(x,\varphi) = e_1 - c^1_{12}(x) \partial_\varphi \\
\mathcal{E}_2(x,\varphi) = E^h_2(x,\varphi) = e_2 - c^2_{12}(x) \partial_\varphi \\
\mathcal{E}_3(x,\varphi) = K(x)\partial_\varphi
\end{array}
\end{equation}
where $c_{ab}^{c}$ are the structure functions of the vector fields on $M$, and  $K(x)$ is the scalar curvature at $x\in M$ of the metric $g$ on $M$.  
\end{definition}

%----------------------section 5: Relationes geometricas
\section{Relation between geometries of $M$ and $P$}

%----------------------- subsection 5.10
\subsection{Structure functions of the frame $\{\mathcal{E}_1,\mathcal{E}_2,\mathcal{E}_3\}$} 

%--------------proposition 5.1
\begin{proposition}\label{prop:5.1}
The structure functions $\hat{c}_{ij}^{k}$ of $\{\mathcal{E}_1,\mathcal{E}_2,\mathcal{E}_3\}$, $i,j,k=\overline{1,3}$,
and the structure functions $c_{ab}^{c}$ of $\{e_1, e_2\}$, $a,b,c=\overline{1,2}$ satisfy 
\begin{equation}
\begin{array}{lll}\label{eq:5_19}
\hat c^1_{12} = c^1_{12} & \hat c^1_{13} = 0 & \hat c^1_{23} = 0  
\\
\hat c^2_{12} = c^2_{12} & \hat c^2_{13} = 0 & \hat c^2_{23} = 0  
\\
\hat c^3_{12} = -1 & \hat c^3_{13} = \frac{e_1 K}{K} & \hat c^3_{23} = \frac{e_2 K}{K}  
\end{array}
\end{equation}
\end{proposition}

%---------------
\begin{proof}
By definition, we have
\begin{equation}
[\mathcal{E}_i,\mathcal{E}_j] = \hat{c}_{ij}^{k}\mathcal{E}_k \, \quad 
[e_a, e_b] = c_{ab}^{c} e_c \quad i,j,k=1,2,3; a,b,c=1,2 \notag
\end{equation}
From \eqref{eq:3_15} and \eqref{eq:3_17}, using \eqref{eq:4_18}, we get that  
\begin{equation}\label{eq:5_20}
[E_{1}^{h}, E_{2}^{h}]=c_{12}^{1}E_{1}^{h}+c_{12}^{2}E_{2}^{h}-K\partial\varphi
\Longrightarrow
[\mathcal{E}_1,\mathcal{E}_2] = c^1_{12}\mathcal{E}_1 + c^2_{12}\mathcal{E}_2 - \mathcal{E}_3.
\end{equation}
By direct calculation we obtain 
\begin{equation}\label{eq:5_21}
[\mathcal{E}_1,\mathcal{E}_3] = [e_1-c^1_{12}(x)\partial_\varphi,K(x)\partial_\varphi] =  e_1 K \partial_\varphi = \frac{e_1 K}{K}\mathcal{E}_3 
\end{equation}
\begin{equation}\label{eq:5_22}
[\mathcal{E}_2,\mathcal{E}_3] = [e_2-c^2_{12}(x)\partial_\varphi,K(x)\partial_\varphi] =  e_2 K \partial_\varphi = \frac{e_2 K}{K}\mathcal{E}_3
\end{equation}
this proves the proposition.
\end{proof}

%----------------------- subsection 5.11
\subsection{The connection coefficients}\label{par:coeficientes-conexion}

In the same way as we got \eqref{eq:3_8}, we can obtain a similar expression for the coefficients of the connection $\hat\nabla$ on $P$ by setting   $X=\mathcal{E}_{1}$,  $Y=\mathcal{E}_{2}$, and  $Z=\mathcal{E}_{3}$ in  \eqref{eq:3_7}: 
\begin{equation}\label{eq:5_23}
\hat{\Gamma}_{ij}^{k} = \frac{1}{2}(\hat{c}_{ij}^{k} + \hat{c}_{ki}^{j} + \hat{c}_{kj}^{i}).
\end{equation}

%------------proposition 5.2
\begin{proposition}\label{prop:5_2}
The connection coefficients of $\hat\nabla$ on $P$ are written as follows 
\begin{equation}
\begin{array}{lll}\label{eq:5_24}
\hat \Gamma^1_{12} = c^1_{12} & \hat \Gamma^1_{13} = 0 & \hat \Gamma^1_{13} = -\frac{1}{2} \\
\hat \Gamma^1_{22} = c^2_{12} & \hat \Gamma^1_{23} = \frac{1}{2} & \hat \Gamma^2_{23} = 0 \\
\hat \Gamma^1_{32} = \frac{1}{2} & \hat \Gamma^1_{33} = \frac{e_1 K}{K} & \hat \Gamma^2_{33} = \frac{e_2 K}{K}
\end{array}
\end{equation}
\end{proposition}
Here $K=K(x)$ is the curvature of $(M,g)$, and $e_{i}K$ means that the vector field $e_{i}$ is applied to the scalar field  $K(x)$.
\begin{proof}
We substitute \eqref{eq:5_19} in \eqref{eq:5_23}, then  get \eqref{eq:5_24}.
\end{proof}

%----------------------- subsection 5.12
\subsection{Coordinates of the curvature tensor}

Using expression \eqref{eq:3_11} for the components of the curvature tensor in terms of the connection coefficients and the structure functions, we get the following expression for the curvature tensor coordinates: 
\begin{equation}\label{eq:5_25}
\hat{R}_{ijk}^l = \mathcal{E}^{h}_i\hat{\Gamma}^{l}_{jk} - \mathcal{E}^{h}_j\hat{\Gamma}^{l}_{ik} + \hat{\Gamma}^{l}_{is}\hat{\Gamma}^{s}_{jk} 
- \hat{\Gamma}^{l}_{js}\hat{\Gamma}^{s}_{ik} - \hat{c}^{s}_{ij} \hat{\Gamma}^{l}_{sk}
\end{equation}

%--------------proposition 5.3
\begin{proposition}\label{prop:5_3}

The coordinates $\hat{R}^{l}_{ijk}$ of the curvature tensor are written as follows:
\begin{equation}\label{eq:5_26}
\begin{array}{l}
\hat R_{1212} = \frac{3}{4} - K,
\\
\hat R_{1213} = \frac{e_1 K}{K},
\\
\hat R_{1223} = \frac{e_2 K}{K},
\\
\hat R_{1313} = -\frac{1}{4} - e_1(\frac{e_1 K}{K}) - c^1_{12} \frac{e_2 K}{K} 
+ \left(\frac{e_1 K }{K}\right)^2,  
\\
\hat R_{1323} = - e_1(\frac{e_2 K}{K}) + c^1_{12} \frac{e_1 K}{K} 
+ \frac{e_1 K }{K}\frac{e_2 K }{K},
\\
\hat R_{2323} = -\frac{1}{4} - e_2(\frac{e_2 K}{K}) + c^2_{12} \frac{e_1 K}{K} 
+ \left(\frac{e_2 K }{K}\right)^2.  
\end{array}
\end{equation}
\end{proposition}

%--------------
\begin{proof}
We substitute \eqref{eq:4_18}, \eqref{eq:5_22} and \eqref{eq:5_24} in \eqref{eq:5_25}, then get the result.
\end{proof}

%----------------------- subsection 5.13 Geodesicas
\subsection{Geodesics}

In order to understand better the relation between the geometrical properties of a metric $g$ on $M$ and its Wagner lift $\hat g$ on $P$, we will consider the relation between geodesics of $g$ and $\hat g$.  
%-------------------------------theorem 5.1
\label{teo:5_1}
\begin{theorem}
Let $\hat\gamma$ be a geodesic of the connection   $\hat \nabla$ on $P$, written with respect to local coordinates as $x^i = \hat\gamma^i(t)$. 
If we denote by $Q^{i}(t)$ the coordinates of the tangent vector field along the geodesic  $\hat\gamma$, 
\begin{equation}
\frac{d}{dt}\hat\gamma(t) = Q^i(t) \mathcal{E}_i |_{\gamma(t)}, \notag
\end{equation}
then the functions   $Q^i(t)$ satisfy the equations 
\begin{eqnarray}
\frac{dQ^1}{dt} + c^1_{12} Q^1 Q^2 + c^2_{12} (Q^2)^2 + Q^2 Q^3 + \frac{e_1 K}{K}(Q^3)^2 = 0
\label{eq:5_28}
\\ 
\frac{dQ^2}{dt} - c^1_{12} (Q^1)^2 - c^2_{12} Q^1 Q^2 - Q^1 Q^3 + \frac{e_2 K}{K}(Q^3)^2 = 0
\label{eq:5_29}
\\
\frac{dQ^3}{dt} - \frac{e_1 K}{K} Q^1 Q^3 - \frac{e_2 K}{K} Q^2 Q^3  = 0
\label{eq:5_30}
\end{eqnarray} 
\end{theorem}

%----------------------------Proof-prop1
\begin{proof}
If $\hat\gamma$ is geodesic, we have 
\begin{equation}
\nabla_{\dfrac{d\hat\gamma}{dt}}\dfrac{d\hat\gamma}{dt} = \nabla_{\dot{\hat\gamma}}{\dot{\hat\gamma}} =0. \notag
\end{equation}
which gives us
\begin{equation}\label{eq:5_27}
\frac{dQ^k}{dt} + \hat\Gamma^k_{ij}(\gamma(t))Q^i(t) Q^j(t) = 0.
\end{equation} 
Thus, by  (\ref{eq:5_24}), we obtain \eqref{eq:5_28}--\eqref{eq:5_30}.  
\end{proof}

%-----------------------------theorem 5.2
\begin{theorem}\label{teo:5_2}
Let  $\hat\gamma(t)$ be a  geodesic of the  metric $\hat g$ on $P$, $\gamma(t)$  its projection on $M$, $t \in [0,a]$. Then
\begin{enumerate}
\item 
If $\frac{d\hat\gamma}{dt}(t)$ is horizontal en $t_0$, then  
$\hat \gamma$ is a horizontal curve, that is $\hat\gamma$ is tangent to $\mathcal{H}$ for all  $t$.
\item 
There holds the equation:
\begin{equation}
\dfrac{\hat g\left(\mathcal{E}_3,\frac{d\hat\gamma}{dt}(t)\right)}{K(\gamma(t))} =C =\text{const.}\notag
\end{equation}
\item 
The curve $\gamma$ satisfies the differential equation 
\begin{equation}
\nabla_{\frac{d\gamma}{dt}}\frac{d\gamma}{dt} = C K J (\dot\gamma) - C^2 K grad K, \notag
\label{eq:2_28}
\end{equation} 
where $J$ is the operator of the complex structure on $M$ associated with the metric $g$.
\item 
If  $\hat \gamma$ is a horizontal geodesic of the metric $\hat g$ on $P$, then $\gamma$ is a geodesic of the metric $g$ on $M$.
\end{enumerate}
\end{theorem}

%---------------------------------Proof-prop2
\begin{proof}
Let $\hat\gamma (t)$, $t\in [0,\alpha]$,  be a geodesic curve of the metric  $\hat{g}$ on $P$.
\begin{enumerate}
\item 
Since the horizontal distribution $\mathcal{H}$ is orthogonal to the orbits of the subgroup of the symmetry group of the metric, 
it is obvious that, if  $\dot{\hat{\gamma}} (t_{0})\in \mathcal{H}$, then $\dot{\hat{\gamma}} (t)\in \mathcal{H}$, for all $t\in [0, \alpha]$.
\item
Since the vectors $\mathcal{E}_{a}=E_{a}^{h}$, $a=1,2$ in $\mathcal{H}$ are horizontal lifts of the orthonormal frame $e_{a}$,  in $M$, we have that 
\begin{equation}\label{eq:5_32}
Q^{a}(t)=\dfrac{d}{dt}\hat{\gamma}^{a}(t)=\dfrac{d}{dt}\gamma^{a}(t), \quad a=1,2.
\end{equation} 
Then $\frac{d}{dt}\gamma(t) = Q^{a}(t)e_{a}$. On the other hand, assuming that $Q^{3}\neq 0$, we can divide (\ref{eq:5_30}) by $Q^{3}$, then  
\begin{equation}
\dfrac{\dot{Q}^{3}}{Q^{3}}=\dfrac{e_{1}K}{K}Q^{1}+\dfrac{e_{2}K}{K}Q^{2}=\dfrac{d}{dt}\left(\log (K(\gamma (t))\right). \notag
\end{equation}
By integration, we obtain that 
\begin{equation}
\log (Q^{3}(t)) = \log (K(\gamma(t))) + \log C \Longrightarrow \dfrac{Q^{3}(t)}{K(\gamma(t))}=C,\notag
\end{equation}
for all  $t$, where $K(\gamma(t)\neq 0$. Since $Q^{3}(t)=\hat{g}(\mathcal{E}_{3}, \dot{\hat{\gamma}}(t))$, we get the required statement.
\item 
If we substitute $Q^3 = C K$ in  \eqref{eq:5_28} and \eqref{eq:5_29}, and rewrite the equations in terms of the connection coefficients,
we get the following system of differential equations: 
\begin{eqnarray}
&&\dfrac{dQ^{1}}{dt}+\Gamma_{12}^{1}Q^{1}Q^{2}+\Gamma_{22}^{1}Q^{2}Q^{2}=-Q^{2}CK-\dfrac{e_{1}K}{K}\left(CK\right)^{2}, \label{eq:5_33}\\
&&\dfrac{dQ^{2}}{dt}+\Gamma_{11}^{2}Q^{1}Q^{1}+\Gamma_{21}^{2}Q^{2}Q^{1}=Q^{1}CK-\dfrac{e_{2}K}{K}\left(CK\right)^{2}. \label{eq:5_34}
\end{eqnarray}
The right hand side of \eqref{eq:5_33} and \eqref{eq:5_34} can be written in the matrix form:
\begin{equation}
=
\left(
\begin{array}{cc}
0 & -1 \\
1 & 0
\end{array}
\right)
\left(\begin{array}{c}
Q^{1}\\
Q^{2}
\end{array}
\right)
CK
-C^{2} K grad (K). \notag
\end{equation} 
The left hand side of \eqref{eq:5_33} and \eqref{eq:5_34} is simply the covariant derivative $\nabla_{\dot\gamma} \dot\gamma$ (see \eqref{eq:5_32}), therefore these two equations can be written as one equation  
\begin{equation}\label{eq:5_35}
\nabla_{\dot\gamma}\dot\gamma = C\cdot K J(\dot\gamma) - C^{2}\cdot K grad (K)
\end{equation}
where $J$ is the operator of the complex structure on $M$ associated with the metric $g$ on $M$.
\item 
From the considerations above it follows that, if   $\hat\gamma (t)$ is a horizontal geodesic, that is  $Q^{3}(t)=0$, and hence $C=0$, we get, by  
 \eqref{eq:5_35}, that  $\gamma (t)$ is a geodesic on $M$.
\end{enumerate}
\end{proof}

%--------------------------------Observation 5.3
\begin{remark}\label{obs:5_3}
The fact that the horizontal geodesics project onto geodesics on the base is of general nature. 
Let us consider a principal  $G$-bundle  $\pi : P \to M$ and a  $G$-invariant metric $\hat g$ on  $P$. Then  a metric  $\hat g$ induces a metric
$g$ on $M$ such that for all $X, Y \in T_p P$ orthogonal to the fiber passing through $u \in P$, we have $\hat g(X,Y) = g(d\pi X, d\pi Y)$. 
In this case, if a geodesic 
$\hat\gamma$ is orthogonal to the fibre at the one point, then $\hat\gamma$ meets all the other fibers orthogonally and projects onto a geodesic of $g$ on $M$  (\cite{BN}, see also \cite{Montgomery}).  
\end{remark}

\begin{remark}\label{obs:5_31}
On a two-dimensional Riemannian manifold $(M,g)$ the Wong equation (\cite{Montgomery}, Sect. 12.2) is the equation 
\begin{equation}
\nabla_{\dot\gamma} \dot\gamma = -\lambda (i_{\dot\gamma}\Omega)^\#,
\label{eq:5_352}
\end{equation}
where $\Omega$ is the area form, $\lambda$ is a constant, and $\#$ stands for the operation of index rising. 
 
In case the curvature $K$ of $(M,g)$ is constant, from \eqref{eq:5_35} it follows that the projection $\gamma$ of a geodesic $\hat\gamma$ satisfies the equation 
\begin{equation}\label{eq:5_351}
\nabla_{\dot\gamma}\dot\gamma = C\cdot K J(\dot\gamma), 
\end{equation}
which is exactly the Wong equation \eqref{eq:5_352}.

\end{remark}

%----------------- section 6: Ejemplos
\section{Examples}

If the curvature of a metric  $g$ is constant (and nonzero)  the formulas for all the geometrical objects on  $P$ become simpler. Let us consider the cases of positive and negative constant curvatures. 

%---------------------------------- subsetion 6.14: Ejemplo1: Curvatura constante positiva
\subsection{Example 1: Constant positive curvature}
%In order to understand the principal bundle $P$ of positively oriented orthonormal frames on $M=\mathbf{S}^{2}$, let us go step by step  paso a paso las siguientes observationes:

%Let $u=(x, A(\varphi))$ be coordinates on P determined by a local orthonormal frame field on $M$,
%where $x\in \mathbf{S}^{2}$ and 
%\begin{equation}
%A(\varphi) =
%\left(
%\begin{array}{cc}
%\cos\varphi & -\sin\varphi \\
%\sin\varphi & \cos\varphi
%\end{array}
%\right). \notag
%\end{equation}

Let us consider a  2-sphere  $\mathbf{S}^2$ with standard metric, and let $P$ is the bundle of positively oriented orthonormal frames on $M=\mathbf{S}^{2}$.
Note that, for any orthonormal frame $e=\{e_{1}, e_{2}\}$ at $x \in \mathbb{S}^2$,  the frame  $e'=\{e_{1}, e_{2}, x\}$ determines a  matrix $A\in SO(3)$, therefore $P$ is diffeomorphic to $SO(3)$. 

To explain how the group $G=SO(2)$ acts on $P\cong SO(3)$, consider  $SO(2)$ as a subgroup of  $SO(3)$ with the following inclusion $\iota : SO(2) \to SO(3)$: 
\begin{equation}
B=
\left(
\begin{array}{cc}
\cos\varphi & -\sin\varphi \\
\sin\varphi & \cos\varphi
\end{array}
\right)
\Longrightarrow
\iota(B)=
\left(
\begin{array}{ccc}
\cos\varphi & -\sin\varphi & 0\\
\sin\varphi & \cos\varphi & 0 \\
0 & 0 & 1
\end{array}
\right)
 \notag
\end{equation}
This matrix describes the rotation of the frame about the third axis defined by 
$x$, therefore $R_{B}A=AB$, for $A \in SO(3)$, $B \in SO(2)$.

Given a matrix $A\in SO(3)$, its third column defines the point  $x\in \mathbf{S}^{2}$, which can be obtained as $x= A\textbf{k}$, where  $\textbf{k}$ is the third standard coordinate vector of $\mathbf{R}^{3}$, hence the projection  $\pi : P \to \mathbf{S}^2$ can be described as follows:
\begin{equation}
P\cong SO(3), \quad \pi : SO(3)\longrightarrow \mathbf{S}^{2},\quad A\mapsto A\textbf{k}. \notag
\end{equation}
Thus the principal bundle $\pi : P \to \mathbf{S}^2$ is isomorphic to the principal $SO(2)$-bundle $q : SO(3) \to \mathbf{S}^2$, where $q : A \to A \textbf{k}$ and $SO(2)$ acts on $SO(3)$ from the right as follows  $(B, A) \in SO(2)\times SO(3) \to A B$.

The generators of the Lie algebra $\mathfrak{so}(3)$ of the Lie group $SO(3)$ are  
\begin{equation}
\xi_1 =
\left(
\begin{matrix}
0 & 0 & -1 
\\
0 & 0 &  0
\\
1 & 0 & 0 
\end{matrix}
\right),
\quad
\xi_2 =
\left(
\begin{matrix}
0 & 0 & 0 
\\
0  & 0 & -1
\\
0 & 1 & 0 
\end{matrix}
\right),
\quad
\xi_3 =
\left(
\begin{matrix}
0 & -1 & 0 
\\
1 & 0 & 0
\\
0 & 0 & 0 
\end{matrix}
\right). \notag
\end{equation}
The corresponding left invariant vector fields $\{\mathcal{E}_{1}, \mathcal{E}_{2}, \mathcal{E}_{3}\}$ ($\mathcal{E}_k(A) = A \xi_k$, $k=1,2,3$) have respectively the following one-parameter subgroups: 
\begin{eqnarray}
\phi_1(t) &=&
\left(
\begin{matrix}
\cos t & 0  & -\sin t 
\\
  0    & 1 &   0
\\
\sin t & 0 &  \cos t 
\end{matrix}
\right),
\quad
\phi_2 (t)=
\left(
\begin{matrix}
1 & 0 & 0 
\\
  0    & \cos t &   -\sin t
\\
0 & \sin t & \cos t 
\end{matrix}
\right), \nonumber\\
\quad
\phi_3 (t) &=&
\left(
\begin{matrix}
\cos t & -\sin t & 0 
\\
  \sin t    & \cos t &   0
\\
0 & 0 & 1 
\end{matrix}
\right).\notag
\end{eqnarray}
Recall that 
\begin{equation}
[\xi_1,\xi_2]=\xi_3, 
\quad 
[\xi_3,\xi_1]=\xi_2, 
\quad 
[\xi_2,\xi_3]=\xi_1. 
\label{eq:6_36}
\end{equation} 
It is clear that $\mathcal{E}_3 =  A \xi_{3}$ is the fundamental field of the principal bundle $q : SO(3) \to \mathbb{S}^2$.

Now let us find  the horizontal distribution  on  $SO(3)$ determined by the Levi-Civita connection   $\nabla$ of the standard metric  on  $M$. 

Let us take the orthonormal frame  $\{\textbf{i},\textbf{j}\}$ at the point  $\textbf{k} \in \mathbf{S}^2$. 
We will find the horizontal lifts of $\textbf{i}$ and  $\textbf{j}$ at the corresponding point  $I \in SO(3)$.  Let us consider the  curve  $\gamma(t) = (\sin t,0,\cos t)$ on the sphere such that  
tal que $\gamma(0)=\textbf{k}$ and $\frac{d}{dt}\gamma(0)=\textbf{i}$. Then, by the properties of parallel  translation with respect to $\nabla$, we get that the horizontal lift $\gamma^h$ of $\gamma$ with  $\gamma^h(0) = I \in SO(3)$ is $\gamma^h(t) = \phi_1(-t)$. 
Therefore, the  horizontal lift of the vector $\textbf{i}^h$ of the vector  $\textbf{i}$ at $I$ is 
$\frac{d}{dt}\gamma^h(0) = - \xi_1$.   
In the same way one can show that the horizontal lift $\textbf{j}^h$ of the vector $\textbf{j}$ 
is $ -\xi_2$.   
Therefore, the horizontal plane  $\mathcal{H}(I)$ at $I \in SO(3)$ is spanned by the vectors  $\xi_1$, $\xi_2$. 

The group  $SO(3)$ acts transitively  on $\mathbf{S}^2$ by isometries. One can easily show that the induced action of  $SO(3)$ on the total space $P$ of the bundle of the orthonormal frames of $\mathbf{S}^2$ is isomorphic to the left action of  $SO(3)$ on itself. Also, since  $SO(3)$ acts on $\mathbf{S}^2$  via  isometries, and so by  automorphisms of the  Levi-Civita connection, the group  $SO(3)$ acts on $P$ via the automorphism of the  horizontal distribution   $\mathcal{H}$. Thus, $\mathcal{H}$ is a left invariant distribution on
$SO(3)$, hence $\mathcal{H}(A) = span(A \xi_1,A \xi_2)$ for any $A \in SO(3)$. 

For the frame field  $\{\mathcal{E}_1,\mathcal{E}_2,\mathcal{E}_3\}$,
from \eqref{eq:6_36} we get that  
\begin{equation}
c^1_{23} = c^2_{31} = c^3_{12}  = 1,  \notag
\end{equation}
and the other structure functions vanish. 
 
From \eqref{eq:3_13} it follows that, for the nonholonomity tensor  $N$, we have  
$N(\mathcal{E}_1,\mathcal{E}_2)= \mathcal{E}_3$. Therefore $\{\mathcal{E}_1,\mathcal{E}_2,\mathcal{E}_3\}$ is a field of frames orthonormal with respect to the Wagner lift  $\hat g$ of $g$ to $\mathbf{R} P^3$.

From above it follows that the orthonormal frame bundle is isomorphic to the bundle $\mathbf{R} P^3 \to \mathbf{S}^2$, which can obtained from the Hopf bundle by taking the quotient with respect to the standard action of $\mathbb{Z}_2$ on $\mathbf{S}^3$, and  
the metric $\hat g$ is the standard elliptic metric on $\mathbf{R} P^3$, which can be obtained from the standard metric on  $\mathbf{S}^3$.

With this example one can also illustrate Theorem~\ref{teo:5_2}, see (\cite{Montgomery},~Sec. 12.2)  

%----------------------- subsetion 6.15 Ejemplo 2: Curvatura Constante Negativa
\subsection{Example 2: Constant negative curvature}   

Let  $(M,g)$ be the Poincar\'e  model of the Lobachevskii plane, i.\,e. 
\begin{equation}
M = \{(x^1,x^2) \mid x^2 > 0\}\quad ;\quad g = \frac{(dx^1)^2+(dx^2)^2}{(x^2)^2}.\notag
\end{equation}
Let us consider a global orthonormal frame field 
\begin{equation}
e_1 = x^2 \partial_1\,;\quad e_2 = x^2 \partial_2\,, \notag
\end{equation}
where
$\partial_a$, $a=1,2$, is the natural frame for the global coordinate system  $(x^1,x^2)$ on $M$.
The Lie bracket of the vector fields of this frame  is  $[e_1,e_2]=-e_1$, hence the structure functions on the base are  $c^1_{12}=-1$, $c^2_{12}=0$. 

As we have a global coordinate system on  $M$, the principal bundle  $P$ of the positively oriented orthonormal frames is isomorphic to the trivial principal bundle $M \times \mathbf{S}^1 \to M$, and we get the coordinates  $(x^1,x^2,\varphi)$ on $P$. 
Now we can directly apply  \eqref{eq:4_18}, in order to find the orthonormal frame field on $M \times \mathbf{S}^1$ with respect to the metric  $\hat{g}$. 
This frame is 
\begin{equation}
\mathcal{E}_1 = e_1 + \partial_\varphi \, ,\quad \mathcal{E}_2 = e_2\, , \quad  \mathcal{E}_3 = \partial_\varphi. \notag
\end{equation}
We can find the structure equations on $P$ using  \eqref{eq:5_20}, \eqref{eq:5_21} and \eqref{eq:5_22}:
\begin{equation}
[\mathcal{E}_1,\mathcal{E}_2] = -\mathcal{E}_1 - \mathcal{E}_3,\quad 
[\mathcal{E}_2,\mathcal{E}_3] = 0, \quad
[\mathcal{E}_3,\mathcal{E}_1] = 0. \notag
\end{equation}   
Now, by  \eqref{eq:5_26} with $K=-1$, we find that the metric $\hat g$ on  
$P = M \times \mathbf{S}^1$ is a metric of nonconstant sectional curvature.  For example, 
the sectional curvature  $\hat K(\mathcal{E}_1 \wedge \mathcal{E}_2)$ is $7/4$, and  
$\hat K(\mathcal{E}_1 \wedge \mathcal{E}_3)$ is $1/4$. 
Note that in the case  $M = \mathbf{S}^2$ the metric $\hat g$ has constant sectional curvature $1/4$, if we apply the same formulas \eqref{eq:5_26} for $K=1$. 

%--------------------------------------------

\end{document}